\documentclass{article}
\usepackage{microtype}
\usepackage{fontenc}[T1]
\usepackage{enumerate}
\usepackage{amsfonts}
\usepackage{amsmath}
\usepackage{latexsym}
\usepackage{amsthm}
\usepackage[british]{babel}
\newcommand{\cC}{\mathcal C}

\newcommand{\bfw}{{\mathbf w}}

\newcommand{\cV}{\mathcal V}

\newcommand{\cS}{\mathcal S}

\newcommand{\fA}{\mathfrak A}
\newcommand{\fF}{\mathfrak F}
\newcommand{\Aut}{{\operatorname{Aut}\,}}
\newcommand{\PG}{{\mathrm{PG}}\,}
\newcommand{\AG}{{\mathrm{AG}}\,}
\newcommand{\PGL}{{\mathrm{PGL}}\,}
\newcommand{\PgL}{{\mathrm{P\Gamma L}}\,}

\newcommand{\FF}{{\mathbb F}}

\theoremstyle{plain}
\newtheorem{prop}{Proposition}[section]
\newtheorem{theorem}[prop]{Theorem}
\newtheorem{corollary}[prop]{Corollary}

\theoremstyle{definition}
\newtheorem{definition}[prop]{Definition}
\newtheorem{remark}[prop]{Remark}

\title{Families of twisted tensor product codes}
\author{L.~Giuzzi\thanks{Part of this
    research has been performed while a guest of
the Department of Mathematics of Ghent University.}
\and V.~Pepe}

\begin{document}
\maketitle
\begin{abstract}
Using geometric properties of the
variety $\cV_{r,t}$, the image
under the Grassmannian map
of a Desarguesian $(t-1)$-spread of $\PG(rt-1,q)$,
we introduce error correcting codes related to
the twisted tensor product construction, producing several families of
constacyclic codes. We exactly determine the
parameters of these codes and
characterise the words of minimum weight.
\end{abstract}

\section{Introduction}
Linear codes are one of the simplest, yet powerful, methods to add
redundancy to a message in order to provide protection against
transmission errors.
For a general reference on coding theory, including standard
notations, see \cite{MS}.
As customary, we regard
a $q$-ary linear code $\cC$ with parameters $[n,k,d]$ as a subspace of
dimension $k$ of $\FF_q^n$ whose non-zero vectors have Hamming weight
at least $d$.
Recall that the correction capacity of a code, at least
on a first approximation, depends on its minimum distance $d$. Indeed,
the most probable undetected errors are exactly those corresponding to
words of minimum weight; thus, it is in practice quite important to be able to
count and characterise such words.
In general, the complexity of implementation of a
code is tied to the size $q$ of the field under consideration.
This is one of the reasons why some of the
most widely used codes are actually defined over the binary field $\FF_2$,
even if some the actual computations involved in the correction procedure
are performed over algebraic extensions, as in the case of the BCH
construction.
On the other hand, 
some recent applications warrant for the use of non-binary codes
in a `natural way', as, in order
to increase the storage density of data, addressable
units larger than a single bit are often selected.
\par
A code $\cC$ might be determined by either a \emph{generator matrix}
$G$, a matrix whose rows constitute a basis for the subspace $\cC$ of
$\FF_q^n$, or, dually, a \emph{parity check matrix} $H$,
that is a matrix providing a basis for the annihilator of $\cC$ in 
$\FF_q^n\,^{\perp}$.
In this paper we shall adopt the latter approach;
in particular, the code $\cC$ is the kernel of the linear
application induced by $H$.
\par
The link between incidence structures
and coding theory has been very fruitful. Possibly, it has
 been first pointed out in \cite{PLJ};
for a reference on
on the  development of the topic and some of the related
problems we refer to the book \cite{AK}.
Codes derived from geometries have proven themselves
to be interesting for several
reasons, not least the possibility of providing synthetic constructions and
their usually large automorphism group.
\par
In this paper we shall study some codes related to
the Segre embedding and twisted tensor products; see \cite{B,CD}.
Our constructions generalise and strengthen some of the results of \cite{B},
albeit using different techniques.
In Section \ref{s:v} we shall recall some properties of the algebraic
variety $\cV_{r,t}$ and prove that any $t+1$ of its points are in
general position. 
In Section \ref{s:c}, using these results,
some new families of codes $\cC_{r,t}$
will be introduced. We will
determine their parameters and also characterise the words of minimum weight.
All of these codes admit a large subgroup of monomial
automorphisms, isomorphic
to $\PgL(r,q^t)$ and  they are always constacyclic in the
sense of \cite{BA}. We shall also investigate some subcodes
and show that
 by puncturing
in a suitable way it is possible to obtain cyclic codes.

\section{The variety $\cV_{r,t}$ and twisted tensor products}
\label{s:v}

Let $\PG(V,\FF)$ be the projective space defined by the lattice
of subspaces of the vector space $V$ over the field $\FF$ and
write $\PG(n-1,q):=\PG(V,\FF_q)$, where $\dim_{\FF_q} V=n$.
Take $\PG(r_1-1,q)$, $\PG(r_2-1,q),\ldots, \PG(r_t-1,q)$ to be $t$
distinct projective spaces; the \emph{Segre embedding}
\[\sigma:
\PG(r_1-1,q) \times \PG(r_2-1,q) \times \cdots \times \PG(r_t-1,q)
\longrightarrow \PG(r_1r_2 \cdots r_t-1,q) \]
 is the map such that
$\sigma(\mathbf{x}^1,\ldots,\mathbf{x}^t)$ is the vector of all the 
possible products $x_{j_1}^{(1)}x_{j_2}^{(2)}\cdots x_{j_t}^{(t)}$, as
$\mathbf{x}^i=(x_0^{(i)},x_1^{(i)},\ldots,x_{n_i-1}^{(i)})$ varies in
$\PG(r_i-1,q)$.
The image of $\sigma$ is the \emph{Segre variety}
$\Sigma_{r_1;r_2;\ldots;r_t}$: it can be regarded, in some way, as a 
product of projective spaces; see \cite[Chapter 25]{Ht}, \cite[Chapter 9]{Ha}
and \cite[Chapter 2]{H}. In the language of tensor products, 
$\sigma$ is the natural  \emph{morphism} between the varieties
\[ \PG(V_1,q) \times
\PG(V_2,q) \times \cdots \times \PG(V_t,q) \longrightarrow
\PG(V_1\otimes V_2 \otimes \cdots \otimes V_t,q). \]
In this paper, we are interested in
the case $r_1=r_2=\ldots,r_t=r$; for brevity we shall
write $\Sigma_{r^t}$ instead of
$\Sigma_{r_1;r_2;\ldots;r_t}$. Clearly,
$\Sigma_{r^t}\subseteq\PG(r^t-1,q)$.
\par
The \emph{Veronese variety} $\cV(n,d)$ is an algebraic
variety of $\PG({{n+d}\choose{d}}-1,q)$, image of the injective map
\[ v_{n,d}:\PG(n,q)\longrightarrow \PG({{n+d}\choose{d}}-1,q), \]
where
$v_{n,d}(x_0,x_1,\ldots,x_n)$ is the vector of all the monomials of
degree $d$ in $x_0,\ldots,x_n$; for $d=2$, see \cite[Chapter 25]{Ht};
for general $d$, see \cite[Chapters 2,9]{H} and also \cite{CLS}.
It is useful to remember that $\cV(1,d)$ is a \emph{normal
rational curve} of $\PG(d,q)$ and any $d+1$ of its points
happen to be in general position.
The Veronese variety $\cV(r,t)$
and $\Sigma_{r^t}$ are closely related, in the
sense that $\cV(r,t)$ is the image under $\sigma$ of the
\emph{diagonal} of $\PG(r-1,q) \times \PG(r-1,q) \times \cdots
\times \PG(r-1,q)$. \par
Take now the projective space $\PG(r-1,q^t)$ and
let $v\longmapsto v^q$ be the $\FF_q$--linear collineation
of order $t$ induced by the
by the Frobenius automorphism of the extension $[\FF_{q^t}:\FF_q]$.
For any $P\in\PG(r-1,q^t)$,
write
\[ {P}^{\alpha}=\sigma([P,P^q,\ldots,P^{q^{t-1}}]). \]
The image of this correspondence is the variety $\cV_{r,t}$.
It is immediate to see that the $\FF_q$--linear
collineation of order $t$
given by
\[ (p_0\otimes p_1\otimes\cdots\otimes p_{t-1}) \longmapsto
(p_{t-1}^q\otimes p_{0}^q\otimes\cdots\otimes p_{t-2}^q) \]
 fixes  $\cV_{r,t}$ point--wise;
hence, $\cV_{r,t}$ is contained
in a subgeometry $\Omega=\PG(r^t-1,q)$ of $\PG(r^t-1,q^t)$.
It turns out that $\cV_{r,t}$ is, in fact, the complete intersection of
the Segre product $\Sigma_{r^t}$ with $\Omega$.

As an algebraic variety $\cV_{r,t}$ first appeared in \cite{Segre};
it has then been described 
in \cite{LunardonNS} and therein extensively studied.
Recently, in
\cite{V},  an explicit parametrisation for $\cV_{r,t}$
has been determined, leading to
the discovery of some new properties. 
It is convenient to recall here this parametrisation.
Take  $\fF=\{ f: \{0,\ldots, t-1\}\to\{0,\ldots,r-1\} \}$ and write
$P=(x_0,\ldots,x_{r-1})\in\PG(r-1,q^t)$.
Then, there is an
injective map $\alpha:\PG(r-1,q^t)\to\cV_{r,t}\subseteq\PG(r^t-1,q^t)$
sending any $P\in\PG(r-1,q^t)$ to the point $P^{\alpha}\in\PG(r^t-1,q^t)$
whose coordinates consist of
all products of the form 
\[ \prod_{i=0}^{t-1} x_{f(i)}^{q^i} \]
as $f$ varies in $\fF$. 
\par
There is a strong affinity 
between the Veronese variety and $\cV_{r,t}$: take 
$\psi \in \PgL(r,q^t)$ so that $\psi^t=id$ and let
$\cV$ be the image under $\sigma$ of the
elements of type $(v,v^{\psi},\ldots,v^{\psi^{t-1}})$; clearly
$\cV$ is a variety; furthermore, when
$\psi=id$, then
$\cV=\Sigma_{r^t} \cap\PG({{r-1+t}\choose{t}}-1,q^t)$ is a Veronese variety;
if, on the contrary,
$\psi$ is a $\FF_q$--linear collineation of order $t$,
ultimately
determining a subgeometry $\Omega=\PG(r^t-1,q)$,
then $\cV=\Sigma_{r^t} \cap\Omega$
and $\cV=\cV_{r,t}$.
\par
We shall also make use of
the alternative description of $\cV_{r,t}$ from
\cite{LunardonNS}.
A \emph{Desarguesian} (also called \emph{normal}) \emph{spread} 
of $\PG(rt-1,q)$ is projectively equivalent to
a linear representation of $\PG(r-1,q^t)$
in $\PG(rt-1,q)$; see \cite{Segre}. As such, it consists of a
collection $\cS$ of $(t-1)$-dimensional subspaces of $\PG(rt-1,q)$,
each of them  the linear representation of a point of
$\PG(r-1,q^t)$, partitioning the point set of $\PG(rt-1,q)$.
When regarded
on the Grassmannian of all the $(t-1)$-dimensional subspaces of
$\PG(rt-1,q)$, the elements of $\cS$ determine the algebraic variety
$\cV_{r,t}$. The best known example is 
for $r=t=2$:
indeed, the Grassmannian of the lines of a
Desarguesian spread of $\PG(3,q)$ is an elliptic quadric
$\cV_{2,2}=\mathcal{Q}^-(3,q)$; see, for instance, \cite[Section 15.4]{james3}.
\par
More in general, if $\Pi_P\in\cS$ is
linear representation 
of a point $P \in \PG(r-1,q^t)$, then the
image under the Grassmann map of $\Pi_P$ is $P^{\alpha}$. Using this
correspondence, it has been possible to investigate several proprieties of
$\cV_{r,t}$; see \cite{L2,LunardonNS,V}. Here we will recall just some
of them.
As the group $\PgL(r,q^t)$ preserves a Desarguesian
$(t-1)$--spread $\mathcal{S}$ of $\PG(rt-1,q)$, its lifting preserves
$\cV_{r,t}$ and its action on the points of $\cV_{r,t}$ is isomorphic
to the $2$--transitive action of $\PgL(r,q^t)$ on the elements of $\cS$;
see \cite{LunardonNS}. We remark that the aforementioned action is
actually $3$--transitive for $r=2$. 

The group $G=\PGL(r,q^t)$ acts
in a natural way on $M=\PG(r-1,q^t)$, which is both a
$G$--module and an $\FF_{q^t}$-vector space.
The \emph{twisted tensor product}  has been introduced in
\cite{St} to realise a new $G$-module, say $M'$, defined over 
the subfield $\FF_q$ form $M$; this induces a
straightforward embedding of  $\PGL(r,q^t)$
in $\PGL(r^t,q)$.
We briefly recall the construction.
Write the action of $G$ on $M$ as
 $g\cdot P\to gP$, where  $g\in G$ and
$P\in M$. For any automorphism  $\phi$ of $\FF_{q^t}$,
we can define  a new $G$-module,
$M^{\phi}$ with group action $g\cdot P\to g^{\phi}P$;
when $\phi$ is the automorphism $g\to g^{q^i}$,
we shall write $M^{\phi}=M^{q^i}$. Using this notation,
the twisted tensor product of $M$ over $\FF_q$ is 
\[ M'=M\otimes M^{q}\otimes\cdots\otimes M^{q^{t-1}}. \]
Observe that as
\[ \PG(r^t-1,q^t)\simeq\PG(M',\FF_{q^t}) \]
we can regard the Segre product $\Sigma_{r^t}$ as
embedded in the latter projective space.
If we restrict our attention to the points of $\cV_{r,t}$,
we see that
for any $g\in G$ and $P\in\PG(r-1,q^t)$
\[ ({gP})^{\alpha}=\sigma([gP,g^qP^q,\ldots,g^{q^{t-1}}P^{q^{t-1}}])=
g\sigma([P,P^q,\ldots,P^{q^{t-1}}])=g{P}^{\alpha}. \]
This is to say that  $\PGL(r,q^t)$, as embedded in
$\PGL(r^t,q)$, stabilises $\cV_{r,t}$ and its action
on the points of the variety is
the same way as on those of $\PG(r-1,q^t)$.
For this reason,
we  can consider $\cV_{r,t}$ as a geometric realisation
of $\PG(r-1,q^t)$ in the twisted tensor product; in brief we shall call it
the \emph{twisted tensor embedding}
over $\FF_q$ of  $\PG(r-1,q^t)$.
In close analogy,
the image under $\alpha$ of a subgeometry
$\PG(r-1,q^s)$ of $\PG(r-1,q^t)$, where $s|t$, is
the twisted tensor embedding of the Veronese
variety $\cV(r-1,\frac{t}{s})$ defined
on the field $\FF_{q^s}$; this turns out
to be the complete intersection of
$\cV_{r,t}$ with a suitable
$\PG({{r-1+\frac{t}{s}}\choose{\frac{t}{s}}}^s-1,q)$; in particular,
for $s=1$, we get a Veronese variety $\cV(r-1,t)$ defined on
$\FF_{q}$; see \cite{LunardonNS} for  $s=1$
and $r=2$ and \cite{V} for the general case.
\par
The set $\mathcal{R} =\{ \Pi_P | P \in \PG(r-1,q)\}$ is a \emph{regulus}.
There 
are several equivalent descriptions for such a collection of spaces
contained in a Desarguesian spread; 
for our purposes, the most useful is the
following: 
suppose $\Sigma$ to be a $(r-1)$--subspace of $\PG(rt-1,q)$ such
that $\Sigma$ intersects every element of $\cS$ in at most one point;
then, the elements of $\cS$ with non-empty intersection with $\Sigma$
form a regulus $\mathcal{R}$.
\par
As mentioned before,  any $t+1$ points of
the
normal rational curve $\cV(1,t)$
are in general position, that is they span a $t$--dimensional
projective space; in \cite{V} it is proved that also any $t+1$
points of $\cV_{2,t}$ are in general position.
Recently, W. Kantor \cite{veron_indip}
has announced a proof of the same property for the Veronesean $\cV(r-1,t)$
with arbitrary $r$. Using a suitable
adaptation of the arguments he proposed
it is possible to generalise
the result of \cite{V} to $\cV_{r,t}$ for any $r$.
To this aim,
first we prove a suitably adapted version of a result in \cite{veron_indip}.

\begin{theorem}
\label{t:m}
Let $\Pi_0,\Pi_1,\ldots,\Pi_{t-1}$ be subspaces of $\PG(r-1,q^t)$ and
suppose that $P\in\PG(r-1,q^t)$ is not contained in any of them.
Then,
$P^{\alpha}$ is not contained in
$\langle
\Pi_0^{\alpha},\Pi_1^{\alpha},\dots,\Pi_{t-1}^{\alpha} \rangle$.
\end{theorem}
\begin{proof}
  For each $i=0,\ldots,t-1$,
  consider a linear map $\ell_i$ 
  vanishing on $\Pi_i$ but not in $P$. That is to say that
  $\ell_i=0$ is the equation of a hyperplane of $\PG(r-1,q^t)$
  containing the subspace $\Pi_i$ but not the point $P$.
  Clearly, for $\ell_i=\sum_{j=0}^{r-1}a_{ij}x_j$, we have
  $\ell_i^q=\sum_{j=0}^{r-1}a_{ij}^qx_j^q$.
  Let
  \[ L=\prod_{i=0}^{t-1}\ell_i^{q^{i-1}}. \]
  By construction, $L$ vanishes on
  $\Pi_0,\Pi_1,\ldots,\Pi_{t-1}$ but not in $P$. By the parametrisation of
  $\cV_{r,t}$ of \cite{V}, $L$ evaluated on $\Pi_i$ is the same as a
  $\FF_{q^t}$--linear function evaluated on
  $\Pi_i^{\alpha} \subset \cV_{r,t}$; hence, there
  exists a hyperplane $\Lambda$ of $\PG(r^t-1,q^t)$ containing
  $\langle \Pi_i^{\alpha} : i=0,1,\ldots,t-1 \rangle$  but not
  $P^{\alpha}$.
  It is well known that any
  hyperplane of $\PG(r^t-1,q^t)$ intersects a subgeometry $\PG(r^t-1,q)$
  in a  (possibly empty) subspace.
  This completes the proof. 
\end{proof}
The
case in which all of the subspaces reduce to 
a single projective point is of special interest for  the
geometry.
\begin{corollary}
\label{t:v}
Any $t+1$ points of $\cV_{r,t}$ are in general position.
\end{corollary}

\begin{corollary}
\label{c:tp2}
  Suppose $q>t$. Any set of $t+2$ dependent
  points of $\cV_{r,t}$
  is contained in the image under $\alpha$ of a
  subline $\PG(1,q) \subset \PG(r-1,q^t)$.
\end{corollary}
\begin{proof}
Take $t+2$ distinct points $P_0,P_1,\ldots,P_{t},P$
forming a dependent system, and let
\[ \Pi_i:=P_i, \text{for $i=0,\ldots,t-2$},\qquad
\Pi_{t-1}=\langle P_{t-1}, P_t \rangle. \]
If it were
$P \notin \Pi_{t-1}$, then, by Theorem \ref{t:m},
$ P^{\alpha}\notin \langle P_0^{\alpha},P_1^{\alpha},\dots,\Pi_{t-1}^{\alpha}
\rangle$.
However, by hypothesis,
$P^{\alpha} \in \langle
P_0^{\alpha},P_1^{\alpha},\dots,P_{t-1}^{\alpha},P_t^{\alpha} \rangle$
and $\langle
P_0^{\alpha},P_1^{\alpha},\dots,P_{t-2}^{\alpha} ,P_{t-1}^{\alpha},P_t^{\alpha}
\rangle \subseteq \langle
P_0^{\alpha},P_1^{\alpha},\dots,\Pi_{t-1}^{\alpha} \rangle$ --- a
contradiction.  It follows that
the $t+2$ points under consideration must all belong
to the same line $\PG(1,q^t)\subseteq\PG(r-1,q^t)$.
In particular, their image is contained in a $\cV_{2,t} \subseteq
\cV_{r,t}$.  By \cite[Lemma 2.5]{L2}, for $t<q$, any $t+2$ linearly
dependent points of
$\cV_{2,t}$ which are $t+1$ by $t+1$ independent
are the image of
elements of the same regulus in the
Desarguesian $(t-1)$--spread of $\PG(2t-1,q)$; it follows that
they are contained in the image under $\alpha$
of the same subline $\PG(1,q)$.
\end{proof}

\section{The code $\cC_{r,t}$ and its automorphisms}
\label{s:c}
\begin{definition}
\label{d:crt}
  Let $q$ be any prime power.
  For any two integers $r,t$ with $t<q$,
  denote by $\cC_{r,t}$ the code whose parity-check matrix
  $H$ has as columns the coordinate vectors of the points of
  the variety $\cV_{r,t}$.
\end{definition}
\begin{remark}
In Definition \ref{d:crt} we did not specify the field over
which the code is defined.
Clearly,
$\cC_{r,t}$ arises as a subspace of $\FF_{q^t}^{r^t}$ ---
however, by the considerations contained in the previous paragraph,  it can
be more conveniently be regarded  as defined over $\FF_q$,
up to  a suitable collineation; this is what we shall do.
\end{remark}
\begin{remark}
\label{r:01}
The order of the columns in $H$ is 
arbitrary, but once chosen, it 
determines an order for the points of $\cV_{r,t}$.
In particular Definition \ref{d:crt} for $\cC_{r,t}$ makes sense only
up to code equivalence, as a permutation of the columns is not usually
an automorphism of the code.
It will be seen in the latter part of this section, however, that
the most useful orders for the code $\cC_{r,t}$ are those induced by
the action of a cyclic collineation group of $\PG(r-1,q^t)$ ---
either a Singer cycle or an affine Singer cycle.
\end{remark}
In view of Remark \ref{r:01}, it is possible to give the following
definition.
\begin{definition}
The \emph{support} of a word $\bfw\in\cC_{r,t}$ is the set of points
of the variety $\cV_{r,t}$ corresponding to the non-zero components
of $\bfw$.
\end{definition}
In order to avoid degenerate cases, the condition
$t<q$ shall always be silently assumed in the remainder of the paper.
\begin{theorem}
\label{t:c0}
  The code $\cC_{r,t}$ has length $n=\frac{(q^{rt}-1)}{(q^t-1)}$
  and parameters $[n,n-r^t,t+2]$.
\end{theorem}
\begin{proof}
  By construction
  \[ n=|\cV_{r,t}|=|\PG(r-1,q^t)|. \]
  As $\cV_{r,t}\subseteq\PG(r^t-1,q)$ is
  not contained in any hyperplane, the rank of the $r^t\times n$ matrix $H$
  is maximal and, consequently, the dimension of the code
  is $n-r^t$.
  Corollary \ref{t:v} guarantees that any $t+1$ columns of $H$
  are linearly independent; thus, by
  \cite[Theorem 10, page 33]{MS}  the
  the minimum distance of $\cC_{r,t}$
  is always at least $d\geq t+2$.
\par
  The image under $\alpha$ of the canonical subline
  $\PG(1,q)$ of $\PG(r-1,q^t)$
  determines a submatrix $H'$ of $H$
  with many repeated rows; indeed, the points represented in $H'$
  constitute a normal rational curve contained in a subspace of
  dimension $t$; see \cite[Theorem 2.16]{L2}.
  It follows that any $t+2$ such points are necessarily dependent.
  Hence, the minimum distance is exactly $t+2$.
\end{proof}
\begin{remark}
The code $\cC_{2,3}$ is also constructed in \cite[Theorem 2]{B}.
\end{remark}

We now characterise the words of minimum weight in $\cC_{r,t}$.
\begin{theorem}
\label{t:param}
  A word $\bfw\in\cC_{r,t}$ has minimum weight if, and only if,
  its support consists of $t+2$
  points  contained in the image of a subline $\PG(1,q)$.
\end{theorem}
\begin{proof}
  By the proof of Theorem \ref{t:c0}, the image of any $t+2$ points in
  a subline $\PG(1,q)$ gives the support of a codeword of minimum weight.
  \par
  Conversely, let $\bfw\in\cC_{r,t}$ be of weight $t+2$; then,
  the support of $\bfw$ consists of $t+2$ dependent points of
  $\cV_{r,t}$. By Corollary \ref{c:tp2} these are contained in
  the image under $\alpha$ of a subline. The result follows.
\end{proof}

We shall now introduce a second class of codes,
also a generalisation of a construction in \cite{B}.

\begin{definition}
\label{d:crts}
  Let $q$ be any prime power.
  For any three integers $r,t,s$ with $t<q$, $1 < s < t$ and $s|t$,
  denote by $\cC_{r,t}^{(s)}$ the code whose parity-check matrix
  $K$ has as columns the coordinate vectors of the points
  of the subvariety of $\cV_{r,t}$
  image of the points of a subgeometry $\PG(r-1,q^s)$, 
  that is the
  twisted tensor  embedding of a Veronese variety
  $\cV(r-1,\frac{t}{s})$ defined over $\FF_{q^s}$.
\end{definition}
\begin{remark}
\label{r:sm}
The matrix $K$ is a submatrix of the matrix $H$.
\end{remark}

\begin{theorem}
\label{t:c2}
  Let $m=m=\frac{(q^{rs}-1)}{(q^s-1)}$ and suppose $t<m-1$.
  Then, the code $\cC_{r,t}^{(s)}$ has length $m$
  and parameters $[m,m-{{r-1+\frac{t}{s}}\choose{\frac{t}{s}}}^s,t+2]$.
\end{theorem}
\begin{proof}
  By construction
  \[ m=|\PG(r-1,q^s)|. \]
  By \cite[Theorem 2]{V}, the subvariety under consideration
  spans a space of rank ${{r-1+\frac{t}{s}}\choose{\frac{t}{s}}}^s$;
  hence, the rank of
  $K$ is ${{r-1+\frac{t}{s}}\choose{\frac{t}{s}}}^s$; consequently,
  the dimension of the code is
  $m-{{r-1+\frac{t}{s}}\choose{\frac{t}{s}}}^s$.
  The variety under consideration is a linear subvariety of $\cV_{r,t}$,
  that is a section of $\cV_{r,t}$ with a suitable subspace.
  Therefore, by Corollary \ref{t:v},
  any $t+1$ of its columns are linearly independent --- this
  shows that the minimum distance of the code is always at least $t+2$.
  On the other hand, since $\PG(r-1,q^s)$ contains the points
  of a subline $\PG(1,q)$, there are also examples of $t+2$ columns
  which are linearly dependent; thus the minimum distance
  is exactly $t+2$.
\end{proof}
\begin{remark}
  Both the codes $\cC_{r,s}$ and $\cC_{r,t}^{(s)}$
  correspond to a $\FF_q$--representation of $\PG(r-1,q^s)$;
  as such they have the same length. They clearly
  differ in their minimum distance $d$, larger
  in the case of the latter code.
  Indeed, the construction of Theorem \ref{t:c2} might be used to
  produce codes over $\FF_q$ with prescribed minimum distance and
  length. As a way to 
  compare the two codes more in detail, consider the
  function
  \[ \eta(\cC)=(d-1)/(n-k). \]
  By Singleton bound, $0<\eta(\cC)\leq 1$ for
  any code, and $\eta(\cC)=1$ if and only if the code is maximum
  distance separable (MDS). Thus, $\eta$
  provides an insight on the cost in redundancy per error (either
  detected or corrected).
  Under this criterion, in general,
  the codes $\cC_{r,s}$ perform much better than
  $\cC_{r,t}^{(s)}$.
  For example, consider the case $r=3$, $s=3$ and $t=6$.
  Then, $\eta(\cC_{3,3})=0.14$, while $\eta(\cC_{3,6}^{(3)})=0.032$.
  For reference, observe that $\eta(\cC_{3,6})=0.10$.
\end{remark}

An automorphism of a linear code $\cC$ is a
isometric linear transformation of $\cC$ --- in other words,
a linear transformation
of $\cC$ preserving the weight of every word.
The set of
all the automorphisms of a code is a group, denoted as $\Aut\cC$.
We shall now focus on automorphisms of a restricted form.
\begin{definition}
  An automorphism $\theta$ of $\cC$ is called \emph{monomial}
  if it is induced by a matrix which has exactly one non-zero
  entry in each row and in each column.
\end{definition}
It is straightforward to see that any monomial transformation
$\cC\to\cC$ is weight preserving and, thus, an automorphism;
clearly, there might also be automorphisms which are not monomial.
\begin{theorem}
\label{t:pgl}
  Any collineation $\gamma\in \PgL(r,q^t)$ lifts to a
  monomial automorphism
  of $\cC_{r,t}$. In particular, there  is a group
  $G\simeq \PgL(r,q^t)$ such that
  \[ G\leq\Aut\cC_{r,t}. \]
\end{theorem}
\begin{proof}
  Let $H_0,H_1,\ldots,H_{n-1}$ be the columns of $H$, the parity check matrix
  of $\cC_{r,t}$, and suppose $\gamma\in \PgL(r,q^t)$.
  Take $\bfw=(w_0,\ldots,w_{n-1})\in\cC$.
  As $\gamma$ acts on the points of the variety $\cV_{r,t}$
  as a permutation group,  there exists a
  permutation $\widetilde{\gamma}$ of $I=\{0,\ldots,n-1\}$
  and elements $[\gamma,i]\in\FF_q$ with $i\in I$ such that
  $\gamma(H_i)=[\gamma,i]H_{\widetilde{\gamma}(i)}$ as vectors;
  hence
  \[ H_i= [\gamma,\widetilde{\gamma}^{-1}(i)]^{-1}H_{\widetilde{\gamma}^{-1}(i)}. \]
  In particular, 
  \[ \mathbf{0}=\sum_i H_iw_i=\sum_i
  [\gamma,\widetilde{\gamma}^{-1}(i)]^{-1}H_{\widetilde{\gamma}^{-1}(i)}w_i=
  \sum_j H_j [\gamma,j]^{-1}w_{\widetilde{\gamma}(j)}. \]
  Thus, $\gamma$ induces a monomial transformation
  $\hat{\gamma}:\cC\to\cC$.
\end{proof}

We now focus our attention on the special case of cyclic automorphisms
and constacyclic codes.
\begin{definition}
A $q$--ary code $\cC$ is \emph{constacyclic} if there exists
$\beta\in\FF_q$ such that for any $\bfw=(w_0\,w_1\,\ldots w_n)\in\cC$,
\[ {\bfw}^{\rho}:=(\beta w_n\,w_0\,\ldots\,w_{n-1})\in\cC. \]
\end{definition}
Constacyclic codes have been introduced in \cite{BA} as
a generalisation of cyclic codes; for some of their properties see,
for instance, \cite{RZ}.
In \cite{B} it is shown that the codes $\cC_{2,3}$ are constacyclic.
Here, we extend the result to all $\cC_{r,t}$.
\par
Observe first that whenever there is a cyclic
group acting regularly on the columns of the parity check matrix of
a code, then the code is cyclic.
\begin{corollary}
\label{l:cyc}
Let $\cC$ be a code of length $\ell$ with parity-check matrix $H$.
Suppose that
there is a cyclic group  acting regularly on the
columns of $H$. Then, $\cC$ is cyclic.
\end{corollary}
\begin{proof}
  Let $\gamma$ be the generator of the group.
  Under the assumptions, $[\gamma,j]=1$ for any $j$.
  The result follows.
\end{proof}
Let now
 $\omega$ be a generator for a Singer cycle of $\PG(r-1,q^t)$ and
order the columns $H_0,H_1,\ldots,H_n$ of the parity check matrix $H$
so that
\[ H_{i+1}=\omega(H_i), \qquad i=0,\ldots,n-1. \]
By construction, $[\omega,i]=1$ for $i<n$, while, in general,
$\omega(H_n)=[\omega,n]H_0$
with $[\omega,n]\neq 1$. Thus, we have the following theorem.
\begin{theorem}
\label{t:ccc}
  The codes $\cC_{r,t}$ are all constacyclic.
\end{theorem}
\begin{remark}
  The automorphism group $G=\PGL(r,q^s)$
  of the subgeometry $\PG(r-1,q^s)$ acts
  as a permutation group on $\cV({r-1,\frac{t}{s}})$;
  furthermore it induces a linear collineation group
  of the ambient space; see \cite[Theorem 2.10]{CLS}.
  Thus,
  the arguments leading to
  Theorems \ref{t:pgl} and \ref{t:ccc} apply in an almost
  identical way to the codes $\cC_{r,t}^{(s)}$.  
  In particular, all of the codes $\cC_{r,t}^{(s)}$
  are equivalent to constacyclic codes  and they admit a monomial
  automorphism group isomorphic to $G$.
\end{remark}

As observed above, $\cC_{r,t}$ is not, in general, cyclic;
see also \cite{B}.
None the less, we can always determine a smaller code which is cyclic and
still related with our geometries.
To this purpose, we shall make use of the
affine Singer cyclic group $S$ of $\PG(r-1,q^t)$;
see \cite{BO42,Sn50}.
This is a linear collineation group which
has exactly $3$ orbits in $\PG(r-1,q^t)$:
an hyperplane $\Sigma$, a single point $O$ with $O\not\in\Sigma$
and the points of $AG(r-1,q^t)=\PG(r-1,q^t)\setminus\Sigma$ different from $O$.
The action on the latter orbit is regular.
These groups turn out to be quite useful in a coding theory setting;
see, for instance, \cite{GLS}.
\par
We recall that
\emph{puncturing} a code $\cC$ means deriving a new code $\cC^*$ from $\cC$
by deleting some of its coordinates;
in general, this procedure decreases the minimum distance;
see \cite[page 28]{MS}.
\par
Let $\xi_i:x_i=0$ be a coordinate hyperplane in $\PG(r-1,q^t)$; with this
choice, the image under $\alpha$ of $\xi_i$ is the full intersection 
of $\cV_{r,t}$ with a suitable hyperplane $\Xi_i$ of $\PG(r^t-1,q)$.
\begin{definition}
Take $O\in\PG(r-1,q^t)$ and suppose $O\not\in\xi_i$.
Write $\widetilde{\cC}_{r,t,i}^{O}$ for the code obtained by puncturing
$\cC_{r,t}$ in the positions corresponding to $O^{\alpha}$ and $\Xi_i$.
\end{definition}
The columns of $\widetilde{\cC}_{r,t,i}^{O}$ correspond
to points in an affine geometry $\AG(r^t-1,q)=\PG(r^t-1,q)\setminus\Xi_i$.
\begin{theorem}
  The code $\widetilde{\cC}=\widetilde{\cC}_{r,t,i}^{O}$
  is equivalent to a cyclic code.
\end{theorem}
\begin{proof}
 Clearly, the length of $\widetilde{\cC}$ is $m=q^{rt-t}-1$.
 Write $\widetilde{H}=(\widetilde{H}_1,\ldots,\widetilde{H}_{\ell})$
 for its parity-check matrix.
 Recall that the columns of $\widetilde{H}$
 correspond to the points of an affine
 geometry $\fA=\AG(r-1,q^t)$ with a point removed.
 By construction, the image of $\fA$ under $\alpha$ is contained
 in an affine subspace $\AG(r^t-1,q)=\PG(r^t-1,q)\setminus\Xi_i$.
 In particular, the affine cyclic Singer group with
 generator $\theta$ lifts to an affine group of
 $\AG(r^t-1,q)$ which
 acts
 cyclically on the columns of $\widetilde{H}$.
 We can assume, up to code equivalence,
 $\theta(\widetilde{H}_i)=H_{i+1}$ for $i<\ell$ and
 $\theta(\widetilde{H}_\ell)=H_0$.
 By Corollary \ref{l:cyc}, $\widetilde{\cC}$ is cyclic.
\end{proof}
The matrix $\widetilde{H}$ contains the coordinates of
affine points such that any $(t+1)$ of them are in general position.
This is to say that any $t$ columns of $\widetilde{H}$ are independent;
thus the minimum distance of the new code is $t+1$.
It is straightforward to see that $\widetilde{\cC}$
it admits a group of  automorphisms isomorphic to
$\operatorname{\Gamma L}\,(r-1,q^t)$.
\par
As a special case, remark that as $\PGL(2,q^t)$
acts $3$-transitively on $\PG(1,q^t)$, the code
$\widetilde{\cC}_{2,t}$ for $r=2$ is the
code $\cC_{2,t}$ punctured in any two of its positions.
\par

\section*{Acknowledgement}
The authors wish to thank W.M. Kantor, for having kindly shared an
unpublished result which provided the insight necessary for obtaining
the proof of Theorem \ref{t:m} in the most general case.

\penalty-10
\vskip.5cm\noindent {\em Authors' addresses}:\\
\penalty10000
\noindent\vskip.2cm
\penalty10000
\begin{minipage}[t]{6cm}
Luca GIUZZI \\
Department of Mathematics \\
Facolt\`a di Ingegneria \\
Universit\`a degli Studi di Brescia \\
Via Valotti 9 \\
I-25133 Brescia (Italy) \\
E--mail: {\tt giuzzi@ing.unibs.it} \\
\end{minipage}
\hfill
\begin{minipage}[t]{6cm}
Valentina PEPE \\
Department of Mathematics \\
Universiteit Gent \\
Building S22 \\
Krijgslaan 281 \\
B-9000 Gent (Belgium) \\
E--mail: {\tt valepepe@cage.ugent.be}
\end{minipage}
\end{document}